\documentclass[12pt]{amsart}

\usepackage{amssymb, amsfonts}
\usepackage{url}
\usepackage[all,arc]{xy}
\usepackage{amscd}
\usepackage{mathrsfs}
\usepackage{geometry}
\usepackage[colorlinks,citecolor=blue]{hyperref}
\usepackage{comment}
\usepackage{enumerate}
\usepackage{graphicx}
\usepackage{bm}
\usepackage{color}
\usepackage{anyfontsize}
\usepackage{caption}

\numberwithin{equation}{section}

\theoremstyle{plain}
\newtheorem{theorem}{Theorem}[section]

\newtheorem{lemma}[theorem]{Lemma}

\newtheorem{definition}[theorem]{Definition}
\newtheorem{remark}[theorem]{Remark}
\newtheorem{example}[theorem]{Example}

\begin{document}

\title[Some rigidity results for static three-manifolds with boundary]{Some rigidity results for static three-manifolds with boundary and positive scalar curvature}

\author{Vladimir Medvedev}

\address{Faculty of Mathematics, National Research University Higher School of Economics, 6 Usacheva Street, Moscow, 119048, Russian Federation}

\email{vomedvedev@hse.ru}

\begin{abstract}
This paper studies three-dimensional compact static manifolds with boundary and positive scalar curvature. We prove that, under a suitable bound on the Ricci curvature, the orientable  quotient of the Nariai static manifold with boundary $Nar_{-1,1}(\mathbb S^2)$ is the only such manifold with connected boundary, provided that the zero-level set of the potential is connected and does not intersect the boundary. We also establish a rigidity theorem for the upper hemisphere with the standard static potential, in the spirit of Cruz and Nunes. 
 \end{abstract}

\maketitle


\newcommand\cont{\operatorname{cont}}
\newcommand\diff{\operatorname{diff}}

\newcommand{\dvol}{\text{dA}}
\newcommand{\Ric}{\operatorname{Ric}}
\newcommand{\Hess}{\operatorname{Hess}}
\newcommand{\GL}{\operatorname{GL}}
\newcommand{\myO}{\operatorname{O}}
\newcommand{\myP}{\operatorname{P}}
\newcommand{\eye}{\operatorname{Id}}
\newcommand{\myF}{\operatorname{F}}
\newcommand{\Vol}{\operatorname{Vol}}
\newcommand{\odd}{\operatorname{odd}}
\newcommand{\even}{\operatorname{even}}
\newcommand{\ol}{\overline}
\newcommand{\mye}{\operatorname{E}}
\newcommand{\myo}{\operatorname{o}}
\newcommand{\myt}{\operatorname{t}}
\newcommand{\irr}{\operatorname{Irr}}
\newcommand{\mydiv}{\operatorname{div}}
\newcommand{\curl}{\operatorname{curl}}
\newcommand{\re}{\operatorname{Re}}
\newcommand{\im}{\operatorname{Im}}
\newcommand{\can}{\operatorname{can}}
\newcommand{\scal}{\operatorname{scal}}
\newcommand{\tr}{\operatorname{trace}}
\newcommand{\sgn}{\operatorname{sgn}}
\newcommand{\SL}{\operatorname{SL}}
\newcommand{\myspan}{\operatorname{span}}
\newcommand{\mydet}{\operatorname{det}}
\newcommand{\SO}{\operatorname{SO}}
\newcommand{\SU}{\operatorname{SU}}
\newcommand{\specl}{\operatorname{spec_{\mathcal{L}}}}
\newcommand{\fix}{\operatorname{Fix}}
\newcommand{\id}{\operatorname{id}}
\newcommand{\grad}{\operatorname{grad}}
\newcommand{\singsup}{\operatorname{singsupp}}
\newcommand{\wave}{\operatorname{wave}}
\newcommand{\ind}{\operatorname{ind}}
\newcommand{\mynull}{\operatorname{null}}
\newcommand{\inj}{\operatorname{inj}}
\newcommand{\arcsinh}{\operatorname{arcsinh}}
\newcommand{\Spec}{\operatorname{Spec}}
\newcommand{\Ind}{\operatorname{Ind}}
\newcommand{\Nul}{\operatorname{Nul}}
\newcommand{\inrad}{\operatorname{inrad}}
\newcommand{\mult}{\operatorname{mult}}
\newcommand{\Length}{\operatorname{Length}}
\newcommand{\Area}{\operatorname{Area}}
\newcommand{\Ker}{\operatorname{Ker}}
\newcommand{\floor}[1]{\left \lfloor #1  \right \rfloor}

\newcommand\restr[2]{{
  \left.\kern-\nulldelimiterspace 
  #1 
  \vphantom{\big|} 
  \right|_{#2} 
  }}


\section{Introduction}

A Riemannian manifold $(M, g)$ with boundary $\partial M$ is called a \textit{static manifold with boundary} if there exists a non-zero function $V \in C^\infty(M)$, called the \textit{(static) potential}, satisfying the following boundary value problem:
\begin{equation}\label{static}
\left\{
\begin{array}{rcl}
\Hess_{g} V - (\Delta_{g} V)\,g - V \Ric_{g} &=& 0 \quad \text{in } M, \\
\noalign{\vskip6pt}
\dfrac{\partial V}{\partial \nu}\,g - V B_{g} &=& 0 \quad \text{on } \partial M.
\end{array}
\right.
\end{equation}
Here, $\nu$ denotes the outward unit normal vector field to $\partial M$, and $B_g$ is the second fundamental form of $\partial M$ with respect to $\nu$. Our sign convention for $B_g$ is such that the unit sphere in Euclidean space has positive mean curvature with respect to the outward unit normal vector field. The triple $(M, g, V)$ is referred to as a \textit{static manifold with boundary}.

The metric on static manifolds with boundary arises in the study of prescribed scalar curvature on $M$ and prescribed mean curvature on $\partial M$, where it is referred to as a non-generic metric (see~\cite{ho2020deformation,cruz2023critical,sheng2024localized,sheng2025static}). The term ``static manifold with boundary" was introduced in~\cite{almaraz2022rigidity}, and the geometric properties of such manifolds have since been investigated in~\cite{sheng2025static,cruz2023static,medvedev2024static}.

Taking the metric trace in system~\eqref{static} the reader can see that~\eqref{static} implies
\begin{equation*}
\left\{
   \begin{array}{rcl}
\Delta_gV&=&-\dfrac{R_g}{n-1}V\quad\mbox{in}\quad M,\\
\dfrac{\partial V}{\partial \nu}&=&\dfrac{H_g}{n-1}V\quad\mbox{on}\quad\partial M.
\end{array}
   \right.
\end{equation*}
 Moreover, as it was shown in~\cite{cruz2023static}, for static manifolds with boundary $R_g=const$ and $H_g=const$. In other words, $V$ satisfies the \textit{Robin problem}. 

In this paper, we focus on compact static manifolds with boundary and positive scalar curvature. Examples of such manifolds include spherical caps, the Schwarzschild-de Sitter static manifold with boundary, and the Nariai static manifold with boundary (see Section 3 in~\cite{medvedev2024static} for further examples). In these examples, one observes two distinct scenarios: either the zero-level set of the potential intersects the boundary, and the boundary is connected (as in the case of spherical caps), or the zero-level set is connected and does not intersect the boundary, but the boundary consists of two connected components. This leads to the following natural question:

\medskip

\textit{Does there exist a compact static manifold with connected boundary and positive scalar curvature such that the zero-level set of the potential is connected and does not intersect the boundary?}

\medskip

The answer to this question is affirmative. An explicit example is constructed as follows:

\medskip

\begin{example}\label{example}

Consider the following \textit{Nariai static manifold with boundary} (see the notation in \textit{Example 10} in~\cite{medvedev2024static}):  
\begin{gather*}
Nar_{-1,1}(\mathbb S^{2}):=\left(\left[-\dfrac{\pi}{2\sqrt{3}},\dfrac{3\pi}{2\sqrt{3}}\right]\times \mathbb S^{2}, g=dr^2+\frac{1}{3}g_{\mathbb S^{2}},V(r)=\frac{1}{\sqrt{3}}\sin \left(\sqrt{3}r\right)\right).
\end{gather*}
The zero-level set of the potential, $\Sigma = V^{-1}(0)$, consists of two connected components corresponding to $r = 0$ and $r = \dfrac{\pi}{\sqrt{3}}$, each of which is a round sphere. Consider the following map
$$
A\colon (r,x)\mapsto \left(\frac{\pi}{\sqrt{3}}-r,-x\right),
$$
where $-x$ stands for the antipodal point to $x\in \mathbb S^{2}$. It is an isometric involution of $\left(\left[-\dfrac{\pi}{2\sqrt{3}}, \dfrac{3\pi}{2\sqrt{3}}\right]\times \mathbb S^{2}, g\right)$ without fixed points and $V\circ A=V$. Then the quotient manifold $Nar_{-1,1}(\mathbb S^{2})/A$, endowed with the quotient metric, is a static manifold with boundary with the potential $V$. It is diffeomorphic to $\mathbb {RP}^3$ minus a ball. Hence, it is orientable and the boundary is connected. The zero-level set of the potential is $\mathbb S^{2}$ with the standard metric. It is connected and does not intersect the boundary. Finally, the scalar curvature of $Nar_{-1,1}(\mathbb S^{2})/A$ is equal to 6 and the area of the zero-level set of the potential equals $\dfrac{4\pi}{3}$.
 
 \begin{remark}\label{nar}
 In fact, there are exactly two manifolds with boundary, admitting a two-sheeted covering with a cylinder $[0,1]\mathbb \times \mathbb S^2$ as the total space. This follows from a straightforward analysis of the $\mathbb Z_2$-action on the total space (see also the remark in the end of Section 7 in~\cite{ambrozio2017static}). Notice, that only one of these two manifolds is orientable. Hence, $Nar_{-1,1}(\mathbb S^2)/A$ is the only orientable static manifold with boundary, admitting a two-sheeted covering with $Nar_{-1,1}(\mathbb S^{2})$ as the total space. 
 \end{remark}
 
 \end{example}

Before proceeding to our first result, we recall the following definition.

\begin{definition}[see~\cite{ambrozio2017static}]\label{def:equiv}
Two static manifolds $(M_{i},g_{i},V_{i})$, $i=1,2$, are said to be \textit{equivalent} if there exists a diffeomorphism $\varphi : M_{1} \rightarrow M_{2}$ such that $\varphi^{*}g_2 = c g_1$ for some constant $c>0$ and $V_2\circ \varphi = \lambda V_1$ for some constant $\lambda$.
\end{definition}

\begin{theorem}\label{them:Ric0}
Let $(M^3, g, V)$ be a compact static manifold with connected boundary, scalar curvature $R_g = 6$, and such that $|\mathring{\mathrm{Ric}}_g|^2 \leqslant 6$. Suppose that $V^{-1}(0) \subset \mathrm{Int}(M)$ is connected. Then $(M^3, g, V)$ is equivalent to $Nar_{-1,1}(\mathbb S^2)/A$.
\end{theorem}

\begin{remark}\label{rem:ifnot}
\begin{enumerate}[(i)]
\item Without assuming that $V^{-1}(0)$ is connected, one can find examples of static manifolds with connected boundary, scalar curvature $R_g = 6$, and $|\mathring{\mathrm{Ric}}_g|^2 \leqslant 6$ that are distinct from $Nar_{-1,1}(\mathbb S^2)/A$. Specifically, consider $Nar_{-k,k}(\mathbb S^{2})$ which is defined as
$$
\left(\frac{1}{\sqrt{3}}\left[\dfrac{\pi}{2}-k\pi, \dfrac{\pi}{2}+k\pi\right]\times \mathbb S^{2}, g=dr^2+\frac13g_{\mathbb S^{2}},V(r)=\frac{1}{\sqrt{3}}\sin \left(\sqrt{3}r\right)\right),
$$ 
for $k \in \mathbb{N}$. Taking the quotient by the involution $A$, as in Example~\ref{example}, yields a static manifold with connected boundary, $R_g = 6$, $|\mathring{\mathrm{Ric}}_g|^2 = 6$, and $H_g=0$ but where $\Sigma = V^{-1}(0)$ has exactly $k$ connected components.
\item The assumption of connected boundary is also essential: without it, there exist static manifolds with boundary that satisfy $R_g = 6$ and $|\mathring{\mathrm{Ric}}_g|^2 \leqslant 6$, yet are not isomorphic to $Nar_{-1,1}(\mathbb S^2)/A$. For instance, the region $Nar_{-1,0}(\mathbb{S}^2)$
$$
\left(\left[-\dfrac{\pi}{2\sqrt{3}}, \dfrac{\pi}{2\sqrt{3}}\right]\times \mathbb S^{2}, g=dr^2+\frac13g_{\mathbb S^{2}},V(r)=\frac{1}{\sqrt{3}}\sin \left(\sqrt{3}r\right)\right),
$$
provides such an example -- it is a compact static manifold with boundary such that $V^{-1}(0)\subset\mathrm{Int}(M)$ is connected, $R_g=6$, $|\mathring{\mathrm{Ric}}_g|^2 = 6$, and $H_g=0$, but its boundary is disconnected.
\item To conclude this remark, we consider a family of examples in which several assumptions from Theorem~\ref{them:Ric0} are dropped.

 Let $ r_h(m) < r_c(m) $ be two positive roots of the function
$$
V_m(r) = \sqrt{1 - r^2 - \frac{2m}{r}}.
$$
\textit{The Schwarzschild-de Sitter static triple} is given by 
$$
\big( [r_h, r_c] \times \mathbb{S}^2,\, g_m = V_m^{-2}\,dr^2 + r^2 g_0,\, V_m \big),
$$
where $ g_0 $ is the standard metric on $ \mathbb{S}^2 $ and $m\in \left(0,\dfrac{1}{3\sqrt{3}}\right)$. Consider the change of variables $ u \colon (0, a) \to (r_h, r_c) $ defined by $ \dfrac{ds}{dr} = V_m(r)^{-1} $, so that the metric takes the form $ g_m = ds^2 + u(s)^2 g_{\mathbb S^2} $. The function $ u $ extends continuously to $ [0, a] $ with $ u(0) = r_h $ and $ u(a) = r_c $, and $ g_m $ extends to a smooth metric on $ [0, a] \times \mathbb{S}^2 $. By reflecting the manifold $ ([0, a] \times \mathbb{S}^2, g_m) $ across its boundary components, we obtain a complete, periodic, rotationally symmetric metric on $ \mathbb{R} \times \mathbb{S}^2 $, which, by abuse of notation, we also denote by $ g_m $ (see the discussion of the Reissner-Nordstr\"om-de Sitter space in~\cite{baltazar2022local}). 

Now let $ r_{ps} = 3m \in (r_h, r_c) $. As shown in~\cite[Section~3]{medvedev2024static}, the sphere $ \{r_{ps}\} \times \mathbb{S}^2 $ is umbilic and has constant mean curvature $ H = 2V_m(r_{ps})/r_{ps} > 0 $. Moreover,  $V_m $ satisfies the equation
$$
\frac{\partial V_m}{\partial \nu} = \frac{H}{2} V_m
$$
on $ \{r_{ps}\} \times \mathbb{S}^2 $, so the second equation of~\eqref{static} holds on this surface.

Let $ s_{ps} = u^{-1}(r_{ps}) $. After performing reflections, we obtain countably many points $ s_i \in \mathbb{R}$ such that the second equation of~\eqref{static} is satisfied on $\{s_i\}\times \mathbb S^2$. Consider three consecutive such points $ s_{i-1} < s_i < s_{i+1} $, where $ s_{i-1} $ and $ s_i $ are separated by a copy of $ r_h $, and $ s_i $ and $ s_{i+1} $ are separated by a copy of $ r_c $. We then obtain two compact static manifolds with two boundary components:
$$
(M_1 = [s_{i-1}, s_i] \times \mathbb{S}^2,g_m, V_m)
\quad\text{and}\quad
(M_2 = [s_i, s_{i+1}] \times \mathbb{S}^2,g_m,V_m).
$$
For $M_1$, the mean curvature of both boundary components is positive; for $M_2$, it is negative, as the outward unit normal points in the direction of decreasing $s$, resulting in a sign change in the second fundamental form. In both cases, the zero-level set of the potential is connected. Moreover, by appropriately choosing the values $s_i$, one can construct static manifolds with two boundary components such that the zero-level set of the potential has an arbitrary number of connected components. Furthermore, the boundary components in such examples can both have positive mean curvature, both have negative mean curvature, or one can be positive and the other negative. For instance, consider the static manifold
$$
([s_{i-1}, s_{i+1}] \times \mathbb{S}^2, g_m, V_m),
$$
where $\{s_{i-1}\} \times \mathbb{S}^2$ has positive mean curvature, $\{s_{i+1}\} \times \mathbb{S}^2$ has negative mean curvature, and the zero-level set of the potential consists of two connected components.

In each of these examples, $R_g=6$, yet the bound $|\mathring{\mathrm{Ric}}_g|^2 \leqslant 6$ is violated.
\end{enumerate}
\end{remark}

\begin{remark}
Theorem~\ref{them:Ric0} also implies the non-existence of a compact static static manifold with connected boundary $(M^3,g,V)$ for which $R_g=6$, $|\mathring{\Ric_g}|^2\leqslant 6$, $V^{-1}(0)\subset  \mathrm{Int}(M)$ is connected, and $H_g\neq 0$.
\end{remark}

For the proof of Theorem~\ref{them:Ric0}, we define the notion of a \textit{Robin static triple}.

\begin{definition} $(M,g,V)$ is a Robin static triple if there exists a positive smooth function $V$, called the potential, which satisfies the following boundary value problem
\begin{equation}\label{static_mixed}
\left\{
   \begin{array}{rcl}
\Hess_{g}V-(\Delta_{g} V){g}-V\Ric_{g} &= &0\quad\mbox{in}\quad M,\\
V& = &0\quad\mbox{on}\quad\partial_D M,\\
\dfrac{\partial V}{\partial \nu}g-VB_{g}& = &0\quad\mbox{on}\quad\partial_R M,
\end{array}
   \right.
\end{equation}
where $\partial M=\partial_D M\sqcup \partial_R M$. 
\end{definition}

The letter~$D$ denotes the Dirichlet boundary condition, and~$R$ denotes the Robin boundary condition. Static triples are a special case of Robin static triples in which $\partial_R M = \varnothing$. When~$\partial_D M = \varnothing$, we obtain a particular case of a static manifold with boundary, in which the potential function~$V$ does not change sign on~$M$. We believe that Robin static triples may be of independent interest and could serve as a useful framework for studying both static triples and static manifolds with boundary. To the best of our knowledge, Robin static triples first appeared implicitly in the literature in~\cite{cruz2023static} (see, for example, Theorem~3). 

One of the ingredients in the proof of Theorem~\ref{thm:impos} is the following technical lemma which is an analogue of Theorem B $(ii)$ in~\cite{ambrozio2017static}.

\begin{lemma}\label{lemma:mixed}
Let $(M^3, g, V)$ be a compact connected Robin static triple with positive scalar curvature. Suppose that at least one connected component of the boundary is topologically a sphere. If the components of $\partial_R M$ have positive mean curvature and at least one component of $\partial_D M$ is not locally area-minimizing, then there is exactly one such component and $\partial_R M = \varnothing$, i.e., $(M^3,g,V)$ is a static triple. 
\end{lemma}

\begin{remark}
Here, we say that a boundary component $\Sigma$ of a manifold $(M,g)$ is locally area-minimizing if there exists an ambient Riemannian manifold $(\widetilde M, \tilde g)$ containing $(M,g)$ isometrically as a subset, such that $\Sigma \subset \mathrm{Int}(\widetilde M)$ and $\Sigma$ is a locally area-minimizing minimal surface in $(\widetilde M, \tilde g)$.\end{remark}

\begin{example}
Consider the Schwarzschild--de Sitter static triple, as defined above. Then $([r_h(m), 3m] \times \mathbb{S}^2, g_m, V_m)$ is a Robin static triple with boundary, in which the Dirichlet part of the boundary $\{r_h(m)\} \times \mathbb{S}^2$ (the black hole horizon) is locally area-minimizing and the Robin part $\{3m\}\times\mathbb{S}^2$ has positive mean curvature. In contrast, $([3m, r_c(m)] \times \mathbb{S}^2, g_m, V_m)$ is another Robin static triple in which the Dirichlet part $\{r_c(m)\} \times \mathbb{S}^2$ (the cosmological horizon) is unstable, but the Robin part $\{3m\}\times\mathbb{S}^2$ has negative mean curvature.
\end{example}

Our second result in this paper is an analogue of Theorem 2 in~\cite{cruz2023static}.

\begin{theorem}\label{CN}
Let $(M^3,g,V)$ be a compact static manifold with boundary such that $R_g=6$. Suppose that $\Sigma=V^{-1}(0)$ is connected. Then
\begin{itemize}
\item[(i)] If $H_g=0$ and $\Sigma\cap \partial M \neq \varnothing$, then $\Sigma$ is a free boundary totally geodesic two-disk and 
\begin{equation*}
|\Sigma|\leqslant 2\pi.    
\end{equation*}
Moreover, in this case, equality holds if and only if $(M^3,g)$ is isometric to the standard spherical cap $(\mathbb S^3_+,g_{\mathbb S^3})$ and $V\in span\{x_1,x_2,x_3\}$, where $x_1,\ldots,x_{4}$ are the coordinates of $\mathbb S^3_+$ in $\mathbb R^4$.
\item[(ii)] If $H_g\geqslant 0$ and $\Sigma\cap \partial M=\varnothing$, then $\Sigma$ is a totally geodesic two-sphere and 
\begin{equation*}
|\Sigma|<4\pi.
\end{equation*}
\end{itemize}
\end{theorem}

The proof of this theorem follows closely the one presented in~\cite[Section 3]{cruz2023static} and relies on an analogue of  \cite[Theorem 1.12]{medvedev2024static} for the case where $\Sigma$ intersects $\partial M$ (see Theorem~\ref{thm:aux} below).

All manifolds considered in this paper are assumed to be orientable.

\subsection{Acknowledgements} The author is grateful to Lucas Ambrozio for his interest in this work and for fruitful discussions. This article is an output of a research project implemented as part of the Basic Research Program at HSE University.
 
\section{Proof of Theorem~\ref{them:Ric0}}

We start with the proof of Lemma~\ref{lemma:mixed}.

\begin{proof}[Proof of Lemma~\ref{lemma:mixed}]

The proof is a straightforward adaptation of the proof of Theorem~B in~\cite{ambrozio2017static} (see also Theorem~10 in~\cite{cruz2024minmax}), so we will be brief.

\medskip

{\bf Claim 1.} \textit{There are no closed minimal surfaces in $\mathrm{Int}(M)$ whose orientable double cover is stable.}

\medskip

Following the argument in part~$(i)$ of Theorem~B, we conclude that the universal cover of $M$ is compact; this relies on the assumption that at least one connected component of $\partial_D M$ is a sphere. Without loss of generality, we may therefore assume that $M$ is simply connected (otherwise, we work on the universal cover). 

Suppose, for contradiction, that $\mathrm{Int}(M)$ contains a closed stable minimal surface $\Sigma$. Since $M$ is simply connected, $\Sigma$ is separating. Let $\partial_n M$ denote the union of components of $\partial_D M$ which are not locally area-minimizing and $\partial_l M$ the locally area-minimizing components of $\partial_D M$. Let $\Omega$ be the connected component of $M \setminus \Sigma$ that contains some components of $\partial_n M$. Minimizing area in the homology class of $\Sigma$ within $\Omega$, we obtain a locally area-minimizing surface $\phi \colon S \to \mathrm{Int}(M)$, as shown in Step~1 of the proof of Theorem~10 in~\cite{cruz2024minmax}. Note that the presence of non-minimal boundary components does not affect the argument, since minimizing sequences cannot approach such components. Now suppose $\phi(S)$ is orientable. Consider the conformal manifold $$(\bar M, \bar g) = (M \setminus \partial M, V^{-2}g).$$ As argued in Section~7 of~\cite{ambrozio2017static}, this space is conformally compact and has bounded geometry. Therefore, for $\delta > 0$ less than the injectivity radius of $(\bar M, \bar g)$, we can define a smooth flow by parallel surfaces $\Phi \colon [0,\delta) \times S \to M$ in $(\bar M, \bar g)$ by
\begin{equation*}
\frac{d}{dt}\Phi_t(x)=V(\Phi_t(x))N_t(x), \quad x\in S, \quad \Phi_0=\phi.
\end{equation*}
Here, $N_t$ denotes the unit normal vector field to the surface $S_t := \Phi_t(S_0)$, where $S_0 = \phi(S)$. For $t \in [0,\delta)$, the surfaces $S_t$ are compact and embedded. Let $H_t$ be the mean curvature of $S_t$. It is well known that
$$
\frac{\partial}{\partial t} H_t = -\Delta_{S_t} V - \left( \mathrm{Ric}_g(N_t,N_t) + |B_{S_t}|^2 \right) V.
$$
Using the relation between $\Delta_g V$ and $\Delta_{S_t} V$, and substituting the first equation in~\eqref{static_mixed}, we obtain
$$
\frac{\partial}{\partial t} H_t = -\langle \nabla^g V, N_t \rangle H_t - |B_{S_t}|^2 V \leqslant -\langle \nabla^g V, N_t \rangle H_t.
$$
Since $H_0=0$, Gr\"onwall's inequality implies $H_t \leqslant 0$ for all $t \in [0,\delta)$. The formula for the first variation of volume then implies that the area $|S_t|$ is non-increasing. But $S_0$ is locally area-minimizing, so $|S_t|$ must be constant. Hence, $H_t = 0$ for all $t \in [0,\delta)$, and consequently $|B_{S_t}|^2 V = 0$, which forces each $S_t$ to be totally geodesic. It follows that $(S_t, g|_{S_t})$ is isometric to $(S_0, g_0)$ for all $t \in [0,\delta)$. 

Now, since $(\bar M, \bar g) = (M \setminus \partial M, V^{-2}g)$ is complete, the flow cannot reach $\partial_D M$ in finite time. Moreover, it cannot touch $\partial_R M$, because $H_g$ is constant and non-zero there, while $H_t = 0$. Thus, the surfaces $S_t$ remain in $\mathrm{Int}(M)$ for as long as the flow exists. Let $T^*$ be the maximal time such that the flow exists and remains smooth. Suppose $T^* < \infty$. Consider a sequence $t_i \to T^*$. Then $\{S_{t_i}\}$ is a sequence of locally area-minimizing surfaces with uniformly bounded area. By the compactness theory for stable minimal surfaces~\cite[Theorem~5.1]{schoen1979existence}, a subsequence converges smoothly to a limit surface $S_{T^*}$, which is also locally area-minimizing. By~\cite[Theorem~5.1]{schoen1979existence}, each $S_{t_i}$ is a topological sphere. Since the convergence is smooth, $S_{T^*}$ must also be a sphere if orientable. However, if $S_{T^*}$ were orientable, the flow could be continued beyond $T^*$, contradicting maximality. Therefore, $S_{T^*}$ must be non-orientable. But the only possible non-orientable surface arising as such a limit would be a topological $\mathbb{RP}^2$, which is impossible, since $M$ is simply connected.

The case where $T^*=\infty$ is treated in the same way as in the proof of Theorem 10 in~\cite{cruz2024minmax}.

\medskip

{\bf Claim 2.} \textit{If one of the components of $\partial_D M$ is not locally area-minimizing, then there is exactly one such component and $\partial_R M = \varnothing$.}

\medskip

We follow the argument in Lemma~3.3 of~\cite{lee2015penrose}. Recall that $\partial_D M = \partial_n M \sqcup \partial_l M$. Let $\Sigma_i$, $i = 1, \dots, l$, be the connected components of $\partial_n M$, and $\Sigma_j'$, $j = 1, \dots, l'$, those of $\partial_l M$. Consider $\Sigma_1$. Minimizing area in its isotopy class yields a surface $\widetilde\Sigma_1$ such that each of its connected components (except possibly those of arbitrarily small area) is parallel to a locally area-minimizing surface. Moreover, $\widetilde\Sigma_1$ is homologous to $\Sigma_1$ in $H_2(M; \mathbb{Z})$ (see~\cite[Theorem~$1'$ and Section~3]{meeks1982embedded}). As established earlier, there are no closed minimal surfaces in $\mathrm{Int}(M)$ whose orientable double cover is stable, and any surface of arbitrarily small area is homologically trivial in $M$. It follows that there exist natural numbers $n_1, \ldots, n_{l'}$ such that
$$
[\Sigma_1] = [\widetilde\Sigma_1] = \sum_{j=1}^{l'} n_j [\Sigma_j'] \quad \text{in } H_2(M; \mathbb{Z}).
$$

Now consider the following segment of the long exact sequence of the pair $(M, \partial M)$:
$$
H_3(M, \partial M) \overset{\partial}{\longrightarrow} H_2(\partial M; \mathbb{Z}) \overset{\iota_*}{\longrightarrow} H_2(M; \mathbb{Z}).
$$
The kernel of $\iota_*$ is generated by the image of the fundamental class:
$$
\partial [M] = \sum_{i=1}^l [\Sigma_i] - \sum_{j=1}^{l'} [\Sigma_j'] - [\partial_R M].
$$

On the other hand, from above we have
$$
[\Sigma_1] - \sum_{j=1}^{l'} n_j [\Sigma_j'] = 0 \quad \text{in } H_2(M; \mathbb{Z}),
$$
so
$$
[\Sigma_1] - \sum_{j=1}^{l'} n_j [\Sigma_j'] \in \ker \iota_*.
$$
Therefore, this class must be an integer multiple of $\partial [M]$. This is only possible if $\partial_R M = \varnothing$, $l = 1$, and $n_j = 1$ for all $j = 1, \dots, l'$. In particular, there is exactly one non-locally-area-minimizing boundary component, and the Robin part of the boundary is empty.
\end{proof}

As an application of Lemma~\ref{lemma:mixed}, we obtain the following lemma.

\begin{lemma}\label{thm:impos}
Let $(M^3,g,V)$ be a static manifold with connected boundary and scalar curvature $6$. If $H_g>0$ and $\Sigma=V^{-1}(0)\subset  \mathrm{Int}(M)$ is connected, then $\Sigma$ is a locally area-minimizing minimal two-sphere and $|\Sigma|<\dfrac{4\pi}{3}$. 
\end{lemma}

\begin{proof}
Consider the connected components of $M \setminus \Sigma$. Let $\Omega$ be the component that does not contain $\partial M$. Then $(\Omega, g, V)$ is a static triple with scalar curvature $6$, and its boundary is $\Sigma$. By the Boucher--Gibbons--Horowitz theorem~\cite{boucher1984uniqueness} and Shen's result~\cite{shen1997note}, it follows that $\Sigma$ is a two-sphere.

Now consider $(M \setminus \Omega, g, V)$, which forms a Robin static triple. Suppose that $\Sigma$, which is equal to $\partial_D (M \setminus \Omega)$, is not locally area-minimizing. Then, by Lemma~\ref{lemma:mixed}, $\partial M$, which is equal to $\partial_R (M \setminus \Omega)$, is empty. This contradicts the assumption that $\partial M$ is non-empty. Therefore, $\Sigma$ must be locally area-minimizing. Since $\Sigma$ is totally geodesic and stable, the stability inequality implies $|\Sigma| \leqslant \dfrac{4\pi}{3}$.

If $|\Sigma| = \dfrac{4\pi}{3}$, then, by~\cite[Theorem~1]{bray2010rigidity}, there exists a neighbourhood $U$ of $\Sigma$ in $(M, g)$ isometric to $(-\varepsilon, \varepsilon) \times \Sigma$ with the product metric. Without loss of generality, we identify $U$ with $(-\varepsilon, \varepsilon) \times \Sigma$. There exists a totally geodesic surface $S \subset (-\varepsilon, 0) \times \Sigma \subset \mathrm{Int}(M \setminus \Omega)$ isometric to $\Sigma$. But $S$ is also locally area-minimizing and, in particular, stable. This contradicts Claim~1 in the proof of Lemma~\ref{lemma:mixed}. Hence, $|\Sigma| < \dfrac{4\pi}{3}$.
\end{proof}

We are now ready to prove Theorem~\ref{them:Ric0}.

\begin{proof}[Proof of Theorem~\ref{them:Ric0}]
Let $\Omega$ be the connected component of $M \setminus \Sigma$ that does not contain $\partial M$, so that $\partial \Omega = \Sigma$. Then $(\Omega, g, V)$ is a static triple with $R_g=6$ and $|\mathring{\mathrm{Ric}}_g|^2 \leqslant 6$. By~\cite[Theorem~A]{ambrozio2017static}, one of the following alternatives holds:

\medskip

\textit{Case 1.} $\mathring{\mathrm{Ric}}_g \equiv 0$ on $\Omega$. By classification, $(\Omega,g,V)$ is equivalent to the standard hemisphere, and in particular, $(\Omega, g)$ is isometric to the upper hemisphere $\mathbb{S}^3_+$ with the standard metric. It follows that $\Sigma = \partial \Omega$ is a totally geodesic sphere and hence unstable. Applying Lemma~\ref{thm:impos}, we conclude that $H_g \leqslant 0$. 

If $H_g=0$, by~\cite[Corollary~4.4]{medvedev2024static}, we have
$$
\int_{M \setminus \Omega} V |\mathring{\mathrm{Ric}}_g|^2 \, dv_g = 0.
$$
But $V$ preserves its sign on $M\setminus\Omega$. Thus, $\mathring{\mathrm{Ric}}_g \equiv 0$ both on $\Omega$ and $M \setminus \Omega$, so $\mathring{\mathrm{Ric}}_g \equiv 0$ on $M$, i.e. $(M,g)$ is an Einstein manifold with boundary. Since its scalar curvature is positive, $(M,g)$ admits a finite covering by a domain on the standard sphere $\mathbb S^3$ which we denote by $\tilde M$. The Frankel argument (see~\cite{frankel1961manifolds,frankel1966fundamental,fraser2014compactness}) implies that $\partial \tilde M$ is connected. Recall that $\partial M$ is also connected by assumption. The contracted Gauss equation implies that $K_{\partial M}=K_{\partial \tilde M}=1$. Thus, by the Gauss-Bonnet theorem, $\partial M$ and $\partial \tilde M$ are the round unit spheres. Hence, the covering is one-sheeted, i.e. $M=\tilde M$. Since $\partial M$ is totally geodesic in $M$, as shown above, $M=\mathbb S^3_+$. We conclude that $M=\mathbb S^3_+$, i.e. it coincides with $\Omega$, which contradicts the assumption.

Consider the case where $H_g<0$. Without loss of generality, assume that $V > 0$ on $M \setminus \Omega$ and $V < 0$ on $\Omega$. Then, by~\cite[Corollary~4.4]{medvedev2024static}, we have
$$
6\int_{M} V\, dv_g>\int_{M}V |\mathring{\Ric_g}|^2\, dv_g=-H_g\left(\frac{H_g^2}{4}+1\right)\int_{\partial M}V\,ds_g.
$$
Observe that for any static manifold $(M^n,g,V)$ with connected boundary, we have
$$
H_g\int_{\partial M}V\,ds_g=(n-1)\int_{\partial M}\dfrac{\partial V}{\partial\nu}\,ds_g=(n-1)\int_M\Delta_gV\,dv_g=R_g\int_MV\,dv_g.
$$
Substituting this to the previous inequality and simplifying, we obtain
$$
H_g\left(\frac{H_g^2}{4}+2\right)\int_{\partial M}V\,ds_g>0.
$$
This yields a contradiction, since $H_g < 0$ and $\displaystyle \int_{\partial M} V  ds_g > 0$.

Therefore, \textit{Case~1} is impossible, so $\mathring{\mathrm{Ric}}_g \not\equiv 0$.

\medskip

\textit{Case 2.} $|\mathring{\mathrm{Ric}}_g|^2 = 6$ on $\Omega$, and $(\Omega, g)$ is covered by the standard cylinder. It follows that $\Omega$ is the orientable quotient of this cylinder and $|\Sigma| = \dfrac{4\pi}{3}$; see the remark at the end of Section~7 in~\cite{ambrozio2017static}. Then, by Lemma~\ref{thm:impos}, $H_g \leqslant 0$: otherwise, $\Sigma$ would be a locally area-minimizing sphere of area $\dfrac{4\pi}{3}$, contradicting the non-existence of such surfaces if $H_g>0$.

Suppose that $|\mathring{\mathrm{Ric}}_g|^2 < 6$ at some point in $M \setminus \Omega$. Without loss of generality, assume that $V > 0$ on $M \setminus \Omega$ and $V < 0$ on $\Omega$. Arguing as in the previous case, we conclude that
$$
H_g\left(\frac{H_g^2}{4}+2\right)\int_{\partial M}V\,ds_g>0,
$$
which is impossible, since $H_g \leqslant 0$ and $\displaystyle \int_{\partial M} V  ds_g > 0$. Therefore, $|\mathring{\mathrm{Ric}}_g|^2 = 6$ on $M \setminus \Omega$, and hence on all of $M$.

Furthermore, by~\cite[Proposition 12]{ambrozio2017static},
$$
0=\left(|\nabla\mathring{\Ric_g}|^2+\frac{|C|^2}{2}\right)V+(R_g|\mathring{\Ric_g}|^2+18\det(\mathring{\Ric_g}))V,
$$
where $C$ denotes the Cotton tensor. By formula (27) in~\cite{ambrozio2017static},
$$
R_g|\mathring{\Ric_g}|^2+18\det(\mathring{\Ric_g}) \geqslant (R_g-\sqrt{6}|\mathring{\Ric_g}|)|\mathring{\Ric_g}|^2V=0.
$$
Substituting this inequality to the previous identity, we obtain that $\mathring{\Ric_g}$ is parallel and $C=0$, i.e., $(M\setminus\Omega,g)$ is conformally flat. Moreover, by \cite[Theorem A (ii)]{ambrozio2017static}, $(\Omega,g,V)$ is covered by a static triple $(\tilde \Omega,\tilde g,\tilde V)$ equivalent to the standard cylinder. In particular, $(\Omega,g)$ is also conformally flat. Hence, $(M,g)$ is conformally flat. 

Consider the covering of $(\Omega, g, V)$ by the static triple $(\widetilde{\Omega}, \widetilde{g}, \widetilde{V})$. This covering is two-sheeted. More precisely, $\Omega=\widetilde\Omega/A$, where $A$ is the map defined in Example~\ref{example}. Now, take two copies of $M \setminus \Omega$ and attach them to $\widetilde{\Omega}$ along the connected components of $\partial \widetilde{\Omega}$. This construction yields a two-sheeted covering space $\widetilde{M}$ of $M$, such that its restriction to $\widetilde{\Omega}$ coincides with the previously defined covering of $\Omega$. Let $\pi \colon \widetilde{M} \to M$ be this covering map. Consider the pullbacks $\tilde{g} = \pi^* g$ and $\tilde{V} = \pi^* V$. Then $(\widetilde{M}, \tilde{g}, \tilde{V})$ is a static manifold with boundary. By construction, the covering is Riemannian -- that is, a local isometry -- so $(\widetilde{M}, \tilde{g})$ inherits the geometric properties of $(M, g)$. As shown above, $(M, g)$ is locally conformally flat, hence so is $(\widetilde{M}, \tilde{g})$. Moreover, on $\widetilde{\Omega}$, the triple $(\widetilde{\Omega}, \tilde{g}, \tilde{V})$ is equivalent to the standard cylinder. By~\cite[Theorem~3.1]{kobayashi1982differential}, the interior of any compact, locally conformally flat static manifold covered by such a model must be isometric to a domain in the standard three-dimensional round cylinder $\mathbb{R} \times \mathbb{S}^2$, otherwise the metric cannot extend smoothly across the boundary. In our case, since $\widetilde{V}$ is (up to scaling) the standard Nariai potential and $\widetilde{V}^{-1}(0)$ is connected, the only possibility is that $(\widetilde{M}, \tilde{g}, \tilde{V})$ is equivalent to $Nar_{-1,1}(\mathbb{S}^2)$. In particular, this implies that the mean curvature $H_g$ of $\partial M$ vanishes. Finally, since the original manifold $(M, g, V)$ is the quotient of $(\widetilde{M}, \tilde{g}, \tilde{V})$ by a fixed-point-free isometric involution $A$, we conclude that $(M, g, V)$ is equivalent to $Nar_{-1,1}(\mathbb{S}^2)/A$ (see Remark~\ref{nar}).

\end{proof}

We finish this section with the following observation. 

\begin{theorem}
Let $(M^3,g,V)$ be a compact orientable locally conformally flat static manifold with connected boundary with positive scalar curvature and $V^{-1}(0)\subset  \mathrm{Int}(M)$. Then it is equivalent to $Nar_{-k,k}(\mathbb S^2)/A$ for some $k\in \mathbb N$.
\end{theorem}

\begin{proof}
Without loss of generality, assume that $R_g=6$. 

Any domain in $M$ bounded by connected components of $V^{-1}(0)$ and not containing $\partial M$ inherits the structure of a static triple $(\Omega, g, V)$. Since such domains are locally conformally flat, it follows from~\cite{kobayashi1982differential,lafontaine1983geometrie} that they must be equivalent to one of the following triples:

1) the standard upper hemisphere
 \begin{equation*}
   \left(\mathbb S^3_{+} , g_{\mathbb S^3} , V = x_4\right);
  \end{equation*}

2) the Schwarzschild--de Sitter manifold
\begin{equation*}
   \left( [r_h(m),r_{c}(m)]\times\mathbb S^2 , g_m = \dfrac{dr^2}{1-r^2-\dfrac{2m}{r}} + r^2 g_{\mathbb S^2}, V_m = \sqrt{1-r^2-\frac{2m}{r}}    \right), 
  \end{equation*}
   where $m\in \left(0,\dfrac{1}{3\sqrt{3}}\right)$ and $r_h(m) < r_c(m)$ are the positive zeroes of $V_m$;

3) the standard cylinder 
 \begin{equation*}
   \left( \left[0,\frac{\pi}{\sqrt{3}}\right] \times \mathbb S^2 , g=dr^2 + \frac{1}{3} g_{\mathbb S^2}, V = \frac{1}{\sqrt{3}} \sin\left(\sqrt{3} r\right) \right)
  \end{equation*} 
  or its quotient by the involution $A\colon (r,x) \mapsto \left(\dfrac{\pi}{\sqrt{3}}-r,-x\right)$. We denote these examples as $Cyl$ and $Cyl/{A}$, respectively.
  
We now analyze how these domains can be glued together along common boundary components. By~\cite[Theorem 3.1]{kobayashi1982differential}, a domain of type~(i) can only be attached to another domain of type~(i), as otherwise the resulting metric would not be smooth.
 
Gluing two domains of the first type yields the closed manifold $\mathbb{S}^3$, which does not have a boundary and thus cannot represent a static manifold with boundary. Therefore, the only admissible configuration is attaching a single domain $\Omega$ containing $\partial M$ to the hemisphere $\mathbb{S}^3_+$. By~\cite[Theorem 3.1]{kobayashi1982differential} by smoothness and local conformal flatness of $(M, g)$, $\Omega$ is isometric to a domain in $\mathbb{S}^3$. Given that $\partial M$ is umbilic, it must be a geodesic sphere. Hence, $\Omega$ is a spherical cap, and $M$ is topologically a 3-sphere minus an open ball. Yet, according to~\cite[Proposition~4.1]{ho2020deformation}, on such a manifold the zero set $V^{-1}(0)$ intersects $\partial M$, contradicting the assumption that $V^{-1}(0) \subset \mathrm{Int}(M)$. 

Consider attaching domains of the second type to each other and to the regions
\begin{equation}\label{domains}
[r_h(m), 3m] \times \mathbb{S}^2 
\quad \text{and} \quad 
[3m, r_c(m)] \times \mathbb{S}^2
\end{equation}
along their corresponding boundary components. We associate a vertex of degree~2 to each domain of the second type, and a vertex of degree~1 to each region in~\eqref{domains}. Each attachment between two domains corresponds to an edge connecting the associated vertices. The resulting graph is connected and consists only of vertices of degree~1 or~2. It is a simple graph-theoretic exercise to show that such a graph must be either a cycle or a path. In the case of a cycle, the resulting manifold is closed, which contradicts our assumption that $(M,g)$ has non-empty boundary. In the case of a path, there are exactly two vertices of degree~1 -- corresponding to the two copies of the regions~\eqref{domains} at the ends -- and thus the resulting manifold has two boundary components. This contradicts the assumption that $\partial M$ is connected. Therefore, no such decomposition can exist under the given conditions.
 
Finally, consider attaching domains of the third type to each other and to the region
\begin{equation}\label{domain}
\left[0,\frac{\pi}{2\sqrt{3}}\right] \times \mathbb{S}^2.
\end{equation}
As in the previous case, we assign a vertex of degree~2 to each domain of type $Cyl$, and a vertex of degree~1 to each copy of $Cyl/{A}$ or of the region~\eqref{domain}. Each attachment corresponds to an edge between the associated vertices. The resulting connected graph has vertices of degree at most~2, so it must be either a cycle or a path. We are interested in the path case, as cycles yield closed manifolds, which are excluded by the assumption that $\partial M \neq \varnothing$. A path has exactly two vertices of degree~1 -- the endpoints -- corresponding to the boundary components of the resulting manifold. If both endpoints correspond to $Cyl/{A}$, the resulting manifold is closed, contradicting $\partial M \neq \varnothing$. If both endpoints are copies of~\eqref{domain}, the resulting manifold has two boundary components, contradicting the connectedness of $\partial M$. Hence, one endpoint must correspond to $Cyl/{A}$ and the other to the region~\eqref{domain}. In this case, the resulting static manifold with boundary is isometric to $Nar_{-k,k}(\mathbb{S}^2)/A$ for some $k \in \mathbb{N}$.
\end{proof}

\section{Proof of Theorem~\ref{CN}} This theorem follows from an analogue of \cite[Theorem 1.12]{medvedev2024static}.

\begin{theorem}\label{thm:aux}
Let $(M^3,g,V)$ be an orientable compact static manifold with boundary with $R_g=6\epsilon,~\epsilon\in\{-1,0,1\}$. Suppose that $\Sigma=V^{-1}(0)$ is connected and $\Sigma\cap \partial M\neq \varnothing$. Let $\Omega$ be a connected component of $M\setminus\Sigma$ with $V>0$ and $S=\partial\Omega\setminus\Sigma$. Then
\begin{equation}\label{general}
\begin{split}
\kappa\left(2\pi\chi(\Sigma)-\left(\epsilon+\frac{H^2_g}{2}\right)|\Sigma|\right)\quad\quad\quad\quad\quad&\\+\left(\frac{3H_g}{2}\int_\Omega V\,dv_g-\int_SV\,ds_g\right)&H_g\epsilon = \int_{\Omega}V|\mathring{\Ric_g}|^2\,dv_g.
\end{split}
\end{equation}
\end{theorem}

\begin{proof}
The formula follows from Schoen's Pohozaev-type integral identity
\begin{equation}\label{schoen}
\dfrac{n-2}{2n}\int_\Omega X(R_g)\,dv_g = -\dfrac{1}{2}\int_\Omega\langle\mathcal{L}_Xg,\mathring\Ric_g\rangle\,dv_g +\int_{\partial \Omega} \mathring\Ric_g(X,\nu)\,ds_g,   
\end{equation}
with $X=\nabla^gV$. One has:
$$
X(R_g)=0, \quad \langle\mathcal{L}_Xg,\mathring\Ric_g\rangle=2V\left(|\Ric_g|^2-\dfrac{R_g^2}{3} \right)=2V|\mathring\Ric_g|^2,
$$
$$
\mathring\Ric_g(X,\nu)=\frac{H_gV}{2}\left(\Ric_g(\nu,\nu)-\frac{R_g}{3}\right)~\quad~\text{on } S,
$$
and
$$
\mathring\Ric_g(X,\nu)=\kappa\left(\frac{R_g}{3}-\Ric_g(\xi,\xi)\right)~\quad~\text{on } \Sigma.
$$
Here, we used that $\mathcal{L}_Xg=2\Hess_gV$, $|\nabla^gV|=\kappa$ and the exterior unit normal $\xi=-\dfrac{\nabla^gV}{|\nabla^g V|}$ on $\Sigma$. Also, $\Ric_g(Y,\nu)=0$ on $S$ for any tangent vector field $Y$ by \cite[Proposition 1 (e)]{cruz2023static}. In particular, $\Ric_g(\nabla^gV,\nu)=\Ric_g(\nu(V)\nu,\nu)$ on $S$. By the contracted Gauss equation,
$$
\Ric_g(\xi,\xi)=3\epsilon-K_\Sigma,
$$
where $K_\Sigma$ is the Gauss curvature of $\Sigma$. Further, it is not difficult to see that
$$
\Delta_gV=\Delta_SV+H_g\nu(V)+\Hess_gV(\nu,\nu)~\quad \text{on } S,
$$
whence
$$
V\Ric_g(\nu,\nu)=-\Delta_SV-H_g\nu(V).
$$
Thus,
$$
\mathring\Ric_g(X,\nu)=\frac{H_g}{2}\left(-\Delta_SV-H_g\frac{\partial V}{\partial \nu}-2\epsilon V\right)~\quad~\text{on } S,
$$
and
$$
\mathring\Ric_g(X,\nu)=\kappa\left(K_\Sigma-\epsilon\right)~\quad~\text{on } \Sigma.
$$

Substituting all this into~\eqref{schoen} and simplifying, we obtain
\begin{equation}\label{penul1}
\begin{split}
0\leqslant \int_\Omega V|\mathring\Ric_g|^2\,dv_g&=\kappa\int_\Sigma(K_\Sigma-\epsilon)\,dv_g\\&-\frac{H_g}{2}\int_S\Delta_SV\,ds_g-\frac{H^2_g}{2}\int_S\frac{\partial V}{\partial \nu}\,ds_g-H_g\epsilon\int_SV\,ds_g.
\end{split}
\end{equation}
Let $\Gamma=\Sigma\cap S$. Observe that the geodesic curvature $k_\Gamma$ of $\Gamma$ in $\Sigma$ is equal to $\dfrac{H_g}{2}$. Indeed, parametrize $\Gamma$ naturally by $t$. Then by definition
$$
k_\Gamma=\langle\nabla^\Sigma_{\dot\Gamma}\dot\Gamma,\nu\rangle_g=\langle\nabla^M_{\dot\Gamma}\dot\Gamma,\nu\rangle_g,
$$
since $B_\Sigma\equiv0$. Here $\nabla^\Sigma$ and $\nabla^M$ denote the Levi-Civita connections of $\Sigma$ and $M$, respectively. From the other side, 
$$
\langle\nabla^M_{\dot\Gamma}\dot\Gamma,\nu\rangle_g=B_{\partial M}(\dot\Gamma,\dot\Gamma)=\frac{H_g}{2},
$$
since $\nabla^{\partial M}_{\dot\Gamma}\dot\Gamma=0$ by  \cite[Proposition 1 (a.2)]{cruz2023static}. Whence, $k_\Gamma=\dfrac{H_g}{2}$.

Further, by the divergence theorem,
\begin{equation}\label{kgamma}
\frac{H_g}{2}\int_S\Delta_SV\,ds_g=\frac{H_g}{2}\int_\Gamma\frac{\partial V}{\partial \xi}\,ds_g=-\kappa\int_\Gamma\frac{H_g}{2}\,ds_g=-\kappa\int_\Gamma k_\Gamma\,ds_g,
\end{equation}
since $\dfrac{\partial V}{\partial \xi}=-\kappa$ everywhere on $\Sigma$.

By the divergence theorem again,
\begin{equation}\label{penul2}
\int_S\frac{\partial V}{\partial \nu}\,ds_g=\int_\Omega\Delta_gV\,dv_g-\int_\Sigma\frac{\partial V}{\partial \xi}\,ds_g=-3\epsilon\int_\Omega V\,dv_g+\kappa |\Sigma|
\end{equation}

Substituting~\eqref{kgamma} and~\eqref{penul2} into~\eqref{penul1} and using the Gauss-Bonnet theorem for surfaces with boundary, we obtain~\eqref{general}.
\end{proof}

\begin{remark}
Formula~\eqref{general} can be rewritten in a form analogous to \cite[formula (1.6)]{medvedev2024static}:
$$
\int_{\Omega}V|\mathring{\Ric_g}|^2\,dv_g+H_g\left(\frac{H_g^2}{4}+\epsilon\right)\int_{S}V\,ds_g= \kappa(2\pi\chi(\Sigma)-\epsilon|\Sigma|).
$$
In fact, if $\Omega$ is a domain whose boundary consists of $d$ connected components $\Sigma_1, \ldots, \Sigma_d \subset V^{-1}(0)$, each intersecting $\partial M$, and $r$ components $S_1, \ldots, S_r \subset \partial M$ with mean curvatures $H_1, \ldots, H_r$, respectively, then a similar formula can be derived:
\begin{align*}
\int_{\Omega}V|\mathring{\Ric_g}|^2\,dv_g+\sum_{j=1}^rH_j\left(\frac{H_j^2}{4}+\epsilon\right)&\int_{S_j}V\,ds_g=\\\nonumber&= \sum_{i=1}^d\kappa_i\left(2\pi\chi(\Sigma_i)-\epsilon|\Sigma_i|\right).
\end{align*}
\end{remark}

\begin{proof}[Proof of Theorem~\ref{CN}]
$(i)$  Taking $H_g=0$ and $\epsilon=1$ in~\eqref{general} in Theorem~\ref{thm:aux}, we obtain
$$
2\pi\chi(\Sigma)-|\Sigma|\geqslant 0.
$$
If equality is achieved, we immediately conclude that $\Sigma$ is a topological disk and $|\Sigma|=2\pi$. Moreover, $\mathring\Ric_g\equiv0$, i.e. $(M,g)$ is an Einstein manifold with boundary. Then, arguing exactly as at the end of \textit{Case 1} in the proof of Theorem~\ref{them:Ric0}, we conclude that $M=\mathbb S^3_+$. By \cite[Proposition 4.1]{ho2020deformation}, $V\in span\{x_1,x_2,x_3\}$, where $x_1,\ldots,x_{4}$ are the coordinates of $\mathbb S^3_+$ in $\mathbb R^4$.

\medskip

$(ii)$ Let $\Omega$ denote the connected component of $M \setminus \Sigma$ that does not contain $\partial M$. Applying~\cite[formula (9)]{ambrozio2017static} (see also \cite[formula (1.6)]{medvedev2024static}) to the static triple $(\Omega, g, V)$, we deduce that $\Sigma$ is a totally geodesic two-sphere and $|\Sigma| \leqslant 4\pi$. The latter inequality also follows from the Boucher-Gibbons-Horowitz and Shen theorem~\cite{boucher1984uniqueness,shen1997note}.

Now suppose $|\Sigma| = 4\pi$. Then applying~\cite[formula (1.6)]{medvedev2024static} to $(M \setminus \Omega, g, V)$, we conclude that $\partial M$ is totally geodesic and $\mathring{\mathrm{Ric}}\equiv 0$. As in case~$(i)$, this implies that $(M, g)$ is isometric to the upper hemisphere $\mathbb{S}^3_+$ with the standard metric. By~\cite[Proposition~4.1]{ho2020deformation}, $V\in span\{x_1,x_2,x_3\}$, where $x_1,\ldots,x_4$ are the coordinates of $\mathbb S^3_+$ in $\mathbb R^4$. In this case, the zero-level set $V^{-1}(0)$ intersects $\partial M$. This contradicts the assumption that the zero-level set of $V$ does not intersect the boundary.
\end{proof}

\bibliography{mybib}

\begin{thebibliography}{CLdS24}

\bibitem[AdL22]{almaraz2022rigidity}
S.~Almaraz and L.~L. de~Lima.
\newblock Rigidity of non-compact static domains in hyperbolic space via
  positive mass theorems.
\newblock {\em arXiv preprint arXiv:2206.09768}, 2022.

\bibitem[Amb17]{ambrozio2017static}
L.~Ambrozio.
\newblock On static three-manifolds with positive scalar curvature.
\newblock {\em Journal of Differential Geometry}, 107(1):1--45, 2017.

\bibitem[BBB22]{baltazar2022local}
H.~Baltazar, A.~Barros, and R.~Batista.
\newblock A local rigidity theorem for minimal two-spheres in an electrovacuum
  spacetime.
\newblock {\em arXiv preprint arXiv:2203.16656}, 2022.

\bibitem[BBN10]{bray2010rigidity}
H.~Bray, S.~Brendle, and A.~Neves.
\newblock Rigidity of area-minimizing two-spheres in three-manifolds.
\newblock {\em Communications in Analysis and Geometry}, 18(4):821--830, 2010.

\bibitem[BGH84]{boucher1984uniqueness}
W.~Boucher, G.~W. Gibbons, and G.~T. Horowitz.
\newblock Uniqueness theorem for anti--de {S}itter spacetime.
\newblock {\em Physical Review D}, 30(12):2447, 1984.

\bibitem[CLdS24]{cruz2024minmax}
T.~Cruz, V.~Lima, and A.~de~Sousa.
\newblock {Min-max minimal surfaces, horizons and electrostatic systems}.
\newblock {\em Journal of Differential Geometry}, 128(2):583 -- 637, 2024.

\bibitem[CN23]{cruz2023static}
T.~Cruz and I.~Nunes.
\newblock On static manifolds satisfying an overdetermined {R}obin type
  condition on the boundary.
\newblock {\em Proceedings of the American Mathematical Society},
  151(11):4971--4982, 2023.

\bibitem[CSS23]{cruz2023critical}
T.~Cruz and A.~Silva~Santos.
\newblock Critical metrics and curvature of metrics with unit volume or unit
  area of the boundary.
\newblock {\em The Journal of Geometric Analysis}, 33(1):22, 2023.

\bibitem[FL14]{fraser2014compactness}
A.~Fraser and M.M.-C. Li.
\newblock Compactness of the space of embedded minimal surfaces with free
  boundary in three-manifolds with nonnegative {R}icci curvature and convex
  boundary.
\newblock {\em Journal of Differential Geometry}, 96(2):183--200, 2014.

\bibitem[Fra61]{frankel1961manifolds}
T.~Frankel.
\newblock Manifolds with positive curvature.
\newblock {\em Pacific Journal of Mathematics}, 11(1):165--174, 1961.

\bibitem[Fra66]{frankel1966fundamental}
T.~Frankel.
\newblock On the fundamental group of a compact minimal submanifold.
\newblock {\em Annals of Mathematics}, 83(1):68--73, 1966.

\bibitem[HH20]{ho2020deformation}
P.~T. Ho and Y.-C. Huang.
\newblock Deformation of the scalar curvature and the mean curvature.
\newblock {\em arXiv preprint arXiv:2008.11893}, 2020.

\bibitem[Kob82]{kobayashi1982differential}
O.~Kobayashi.
\newblock A differential equation arising from scalar curvature function.
\newblock {\em Journal of the Mathematical Society of Japan}, 34(4):665--675,
  1982.

\bibitem[Laf83]{lafontaine1983geometrie}
J.~Lafontaine.
\newblock Sur la g{\'e}om{\'e}trie d'une g{\'e}n{\'e}ralisation de
  l'{\'e}quation diff{\'e}rentielle d'{O}bata.
\newblock {\em J. Math. Pures Appl}, 62(9):63--72, 1983.

\bibitem[LN15]{lee2015penrose}
D.~A. Lee and A.~Neves.
\newblock The {P}enrose inequality for asymptotically locally hyperbolic spaces
  with nonpositive mass.
\newblock {\em Communications in Mathematical Physics}, 339:327--352, 2015.

\bibitem[Med24]{medvedev2024static}
V.~Medvedev.
\newblock Static manifolds with boundary: {T}heir geometry and some uniqueness
  theorems.
\newblock {\em arXiv preprint arXiv:2410.15347}, 2024.

\bibitem[MSY82]{meeks1982embedded}
W.~Meeks, L.~Simon, and S.-T. Yau.
\newblock Embedded minimal surfaces, exotic spheres, and manifolds with
  positive {R}icci curvature.
\newblock {\em Annals of Mathematics}, 116(3):621--659, 1982.

\bibitem[She97]{shen1997note}
Y.~Shen.
\newblock A note on {F}ischer-{M}arsden's conjecture.
\newblock {\em Proceedings of the American Mathematical Society},
  125(3):901--905, 1997.

\bibitem[She24]{sheng2024localized}
H.~Sheng.
\newblock Localized deformation of the scalar curvature and the mean curvature.
\newblock {\em arXiv preprint arXiv:2402.08619}, 2024.

\bibitem[She25]{sheng2025static}
H.~Sheng.
\newblock Static manifolds with boundary and rigidity of scalar curvature and
  mean curvature.
\newblock {\em International Mathematics Research Notices}, 2025(7):rnaf086,
  2025.

\bibitem[SY79]{schoen1979existence}
R.~Schoen and S.-T. Yau.
\newblock Existence of incompressible minimal surfaces and the topology of
  three dimensional manifolds with non-negative scalar curvature.
\newblock {\em Annals of Mathematics}, 110(1):127--142, 1979.

\end{thebibliography}
\bibliographystyle{alpha}

\end{document}